\documentclass[oneside]{amsart}

\usepackage{amssymb}
\usepackage{amscd}

\numberwithin{equation}{section} 
\numberwithin{figure}{section} 
\theoremstyle{plain}
\newtheorem*{thm*}{Theorem}
\theoremstyle{plain}
\newtheorem{thm}{Theorem}[section]
\theoremstyle{definition}

\theoremstyle{plain}
\newtheorem{lem}[thm]{Lemma}
\theoremstyle{plain}

\theoremstyle{plain}

\theoremstyle{remark}

\theoremstyle{remark}
\newtheorem*{acknowledgement*}{Acknowledgement}

\begin{document}

\title[Finsler metrizabilities and geodesic invariance]{Finsler metrizabilities and geodesic invariance}

\author[Bucataru]{Ioan Bucataru}
\address{Faculty of  Mathematics \\ Alexandru Ioan Cuza University \\ Ia\c si, 
  Romania}
\email{bucataru@uaic.ro}
\urladdr{http://orcid.org/0000-0002-8506-7567}

\author[Constantinescu]{Oana Constantinescu}
\address{Faculty of  Mathematics \\ Alexandru Ioan Cuza University \\ Ia\c si, 
  Romania}
\email{oanacon@uaic.ro}
\urladdr{http://orcid.org/0000-0003-2687-2029}

\date{\today}
\begin{abstract}
We prove that various Finsler metrizability problems for sprays can be reformulated in terms of the geodesic invariance of two tensors (metric and angular). We show that gyroscopic sprays is the the largest class of sprays with geodesic invariant angular metric. Scalar functions associated to these geodesically invariant tensors will be invariant as well and therefore will provide first integrals for the given spray.
\end{abstract}



\maketitle

\section{Introduction and main results.}
A symmetric linear connection $\mathcal{D}$ represents the Levi-Civita connection of some Riemannian metric $g$ if and only if the metric tensor is covariant constant with respect to the given connection, $\mathcal{D}g=0$. The problem of deciding if for a given connection there is a covariant constant Riemannian metric is known as the metrizability problem, it has been studied by many authors with various formulations for the integrability conditions of the problem, \cite{Kowalski88, MVV08, Schmidt73, TK12}. In Finsler geometry, the metrizability problem for a linear connection on the total space of a vector bundle has been studied in \cite[Theorem 4.2]{Anastasiei03}. The more difficult problem, known as the projective metrizability problem, requires to decide whether for a given connection, its projective class contains the Levi-Civita connection of some Riemannian metric. This problem has been studied for a very long time \cite{Berwald36, MVV08} and it has been solved recently only in the two-dimensional case, \cite{BDE09}.   

The problem of deciding if a given spray $S$ (a system of second order homogeneous ordinary differential equations) represents the geodesic spray (the geodesic equations) of some Finsler metric is known as the Finsler metrizability problem,  and it has been intensively investigated in the last decades, \cite{BD09, GM00, KS85, SV02}. The more general problem, known as the projective metrizability problem, consists in deciding if the geodesics of a given spray represent the geodesics of a Finsler metric, up to an orientation preserving reparametrisation, \cite{BC20, SV02, KS85}. The case when the spray is flat (has rectilinear geodesics)  has been studied first by Berwald in \cite{Berwald41}. These problems are part of the more general problem, the inverse problem of Lagrangian mechanics, which requires to decide whether or not a given system of second order ordinary differential equations (semispray) represents the variational equations (geodesic semispray) of some Lagrangian function, \cite{BD09, GM00}.    

A rigorous formulation of the metrizability problem requires a geometric setting with corresponding covariant derivatives. In Finsler geometry, there is not a canonical connection associated to a given spray or a Finsler metric, \cite{Grifone72, GM00, Shen01, SLK14}. However, all of these connections have the same covariant derivative along the geodesic spray, which is called the dynamical covariant derivative, \cite{BD09}. For a given spray $S$, the induced geometric framework uniquely determines the dynamical covariant derivative $\nabla$, \cite{BCD11}. 

In this work, we will consider $M$ a connected, $n$-dimensional, smooth manifold, where local coordinates will be denoted by $x^i$. We consider the tangent bundle $TM$ and the slit tangent bundle $T_0M=TM\setminus \{0\}$, where local coordinates will be denoted by $(x^i, y^i)$. A Finsler structure is given by a continuous and positive function $F: TM \to \mathbb{R}$, positively homogeneous in the fibre coordinates, smooth on $T_0M$, such that the metric tensor 
\begin{eqnarray}
g_{ij}(x,y)=\dfrac{1}{2}\dfrac{\partial^2F^2}{\partial y^i \partial y^j}(x,y) \label{gij}
\end{eqnarray}  
is non-degenerate on $T_0M$.

A spray $S\in \mathfrak{X}(T_0M)$ is a second order vector field, homogeneous of order $2$. For a given Finsler metric $F$, there is a unique spray $S$, called the geodesic spray, that satisfies the Euler-Lagrange equations
\begin{eqnarray}
S\left(\dfrac{\partial F^2}{\partial y^i}\right) - \dfrac{\partial F^2}{\partial x^i}=0. \label{elf2}
\end{eqnarray}

Similarly, with the corresponding result in Riemannian geometry, we will prove the following result.
\begin{thm}  \label{thm:fm}
A spray $S$ is the geodesic spray of some Finsler structure $F$ if and only if its metric tensor $g_{ij}$ is covariant constant with respect to the induced dynamical covariant derivative:
\begin{eqnarray}
\nabla g_{ij}=0. \label{ngij}
\end{eqnarray}
\end{thm}
In the Riemannian case, the metric tensor \eqref{gij} is independent of the fibre coordinates, hence the equation \eqref{ngij} is linear in these fibre coordinates: $g_{ij\mid k}y^k=0$. Therefore, in this case, the metrizability conditions \eqref{ngij} reduce to the classic ones: $g_{ij\mid k}=0$. 

For the inverse problem of Lagrangian mechanics, the condition \eqref{ngij} is known as one of the Helmholtz conditions and it is only a necessary condition. We prove that for the Finsler metrizability problem, this Helmholtz condition is also sufficient. 

Regarding the projective metrizability problem, we can formulate it using the covariant derivative of the metric tensor of some Finsler function, similarly with the Levi-Civita equations from the Riemannian geometry.
\begin{thm} \label{thm:pm}
Consider $S$ a spray and a Finsler structure $\widetilde{F}$.
\begin{itemize}
\item[i)] The spray $S$ is projectively related to the Finsler structure $\widetilde{F}$ if and only if there exists a $1$-homogeneous function $P\in C^{\infty}(T_0M)$ such the metric tensor $\widetilde{g}_{ij}$ satisfies the following Levi-Civita equations:
\begin{eqnarray}
\nabla \widetilde{g}_{ij}= 2P \widetilde{g}_{ij} + \dfrac{\partial P}{\partial y^i} \widetilde{g}_{kj}y^k + \dfrac{\partial P}{\partial y^j} \widetilde{g}_{ik}y^k.
\label{npgij} 
\end{eqnarray} 
\item[ii)] If the spray $S$ is projectively related to the Finsler structure $\widetilde{F}$, then we have the following geodesic invariance: 
\begin{eqnarray}
\nabla\left(\frac{1}{\widetilde{F}}\left(\widetilde{g}_{ij} - \frac{\partial \widetilde{F}}{\partial y^i} \frac{\partial \widetilde{F}}{\partial y^j}\right)\right)=0. \label{nhij} 
\end{eqnarray}
\end{itemize}
\end{thm}
In the Riemannian case, the projective function $P$ in linear, $P(x,y)=a_i(x)y^i$, the projective metrizability conditions \eqref{npgij} are linear in the fibre coordinates and therefore they reduce to the classic Levi-Civita equations, \cite[Theorem 5.2]{MVV08}:
\begin{eqnarray*}
 \widetilde{g}_{ij\mid k}= 2a_k \widetilde{g}_{ij} + a_i \widetilde{g}_{kj} + a_j \widetilde{g}_{ik}.
\end{eqnarray*}     
The geodesic invariant tensor from formula \eqref{nhij} can be expressed in terms of the angular metric, \cite[(9.2.20)]{SLK14}:
\begin{eqnarray}
\widetilde{h}_{ij}=\widetilde{g}_{ij} - \frac{\partial \widetilde{F}}{\partial y^i} \frac{\partial \widetilde{F}}{\partial y^j}=\widetilde{F}\frac{\partial^2 \widetilde{F}}{\partial y^i \partial y^j}. \label{angular}
\end{eqnarray}
Therefore, we can reformulate the geodesic invariance equation \eqref{nhij}, in terms of the angular metric \eqref{angular}, as follows:
\begin{eqnarray}
\nabla\left(\frac{1}{\widetilde{F}}\widetilde{h}_{ij}\right) = \nabla \left(\frac{\partial^2 \widetilde{F}}{\partial y^i \partial y^j}\right) = 0. \label{nhij2}
\end{eqnarray}  
The geodesic invariance condition \eqref{nhij}, or the equivalent formulations \eqref{nhij2}, are only necessary conditions for the projective metrizability of a given spray $S$. We consider now the most general class of sprays that can be characterized by the geodesic invariant condition \eqref{nhij}. 

A spray $S$ is called \emph{gyroscopic} if there exists a basic $2$-form $\omega\in \Lambda^2(M)$ and a Finsler structure $\widetilde{F}$ that satisfies the Euler-Lagrange equations, \cite{BC15}:
\begin{eqnarray}
S\left(\frac{\partial \widetilde{F}}{\partial y^i}\right) - \frac{\partial \widetilde{F}}{\partial x^i}=\omega_{ij}y^j. \label{isomega}
\end{eqnarray} 
 
In the next theorem we show that the class of gyroscopic sprays is the largest class of sprays that can be characterized by the geodesic invariance condition \eqref{nhij}.
\begin{thm} \label{thm:gyroscopic}
A spray $S$ is gyroscopic if and only if there exists a Finsler structure $\widetilde{F}$ that satisfies the geodesic invariance condition \eqref{nhij}.
\end{thm}
  
Consider $S$ the geodesic spray of some Finsler function $F$ that is also gyroscopic (or it is projectively related) to some Finsler metric $\widetilde{F}$. Using the geodesic invariances \eqref{ngij} and \eqref{nhij2} it follows that the $(1,1)$-type tensor 
\begin{eqnarray}
\mathcal{H}^i_j=\dfrac{F}{\widetilde{F}}g^{ik}\widetilde{h}_{kj} = \frac{F}{\widetilde{F}^3}g^{ik}\left(\widetilde{g}_{kj}\widetilde{g}_{ls} - \widetilde{g}_{kl}\widetilde{g}_{js}\right)y^ly^s. \label{hij}
\end{eqnarray}
 is geodesically invariant. Consequently, all scalar functions (traces, eigenvalues, coefficients of the characteristic polynomial) determined by the $(1,1)$-type tensor \eqref{hij} are geodesically invariant and hence they provide first integrals for the geodesic spray $S$, \cite{Bucataru22}.  

\section{Finsler metrics and induced geometric structures.}

On an $n$-dimensional manifold $M$, we consider $F:TM \to \mathbb{R}$ a Finsler structure with the metric tensor $g_{ij}$, given by formula \eqref{gij}. The homogeneity condition for the Finsler function $F$ implies, using Euler's theorem, that: 
\begin{eqnarray*}
\frac{\partial F}{\partial y^i}y^i = F, \quad \frac{\partial F^{2}}{\partial y^{i}} = 2g_{ij}(x,y)y^{j}, \quad F^{2}(x,y)=g_{ij}(x,y)y^{i}y^{j}.
\end{eqnarray*}
Together with the metric tensor $g_{ij}$, we consider also \emph{angular metric}:
\begin{eqnarray*}
h_{ij}=F\frac{\partial^{2}F}{\partial y^{i}\partial y^{j}}=g_{ij}-\frac{\partial F}{\partial y^{i}}\frac{\partial F}{\partial y^{j}}
\end{eqnarray*}
The homogeneity of the Finsler function $F$ assures that $h_{ij}y^{j}=0$. Therefore, the regularity condition for the metric tensor $g_{ij}$, $\operatorname{rank}(g_{ij})=n$, is equivalent to $\operatorname{rank}(h_{ij})=n-1$. 

A \emph{spray} $S$ is a globally defined vector field on $T_{0}M$, $S\in\mathfrak{X}(T_{0}M)$, that satisfies:
\begin{itemize}
\item[i)] $JS=\mathbb{\mathcal{C}}$ (it is a second order vector
field) and
\item[ii)] $[\mathbb{\mathcal{C}},S]=S$ (it is homogeneous of order $2$).
\end{itemize} 
Here $\mathcal{C}=y^{i}\frac{\partial}{\partial y^{i}}$ is the Liouville vector field on $T_{0}M$ and $J=\frac{\partial}{\partial y^{i}}\otimes dx^{i}$ is the tangent
structure (or vertical endomorphism). Locally, a spray can be expressed as follows: 
\begin{eqnarray}
S=y^{i}\frac{\partial}{\partial x^{i}}-2G^{i}(x,y)\frac{\partial}{\partial y^{i}},\label{semispray}
\end{eqnarray}
where $G^{i}$ are locally defined, homogeneous functions of order $2$ on $T_{0}M$. 

A regular curve $\gamma:I\to M$ is a geodesic
of $S$ if $S\circ\gamma'=\gamma''$. Locally, a curve
$\gamma(t)=(x^{i}(t))$ is a geodesic of the spray $S$, given by formula \eqref{semispray}, if it satisfies
the system of second order ordinary differential equations 
\begin{eqnarray}
\frac{d^{2}x^{i}}{dt^{2}}+2G^{i}\left(x,\frac{dx}{dt}\right)=0.\label{sode}
\end{eqnarray}
A spray $S$ induces a nonlinear connection on $T_{0}M$, and hence a horizontal distribution $HT_0M$ that is supplementary to the vertical distribution $VT_0M=\operatorname{Ker}J$.  The corresponding
horizontal and vertical projectors, $h$ and $v$, are locally given
by, \cite{Grifone72}, 
\begin{eqnarray*}
h=\frac{\delta}{\delta x^{i}}\otimes dx^{i},\quad v=\frac{\partial}{\partial y^{i}}\otimes\delta y^{i},\quad\frac{\delta}{\delta x^{i}}=\frac{\partial}{\partial x^{i}}-N_{i}^{j}\frac{\partial}{\partial y^{j}},\quad\delta y^{i}=dy^{i}+N_{j}^{i}dx^{j},\quad N_{j}^{i}=\frac{\partial G^{i}}{\partial y^{j}}.
\end{eqnarray*}
A spray $S$ determines also a tensor derivation on $T_{0}M$, called the \emph{dynamical covariant derivative}. Its
action on smooth functions and vector fields is given by, \cite{BD09},
\begin{eqnarray*}
\nabla f=S(f), \quad \forall f\in C^{\infty}(T_{0}M),\quad \nabla X=h\left[S,hX\right]+v\left[S,vX\right],\quad \forall X\in\mathfrak{X}(T_{0}M).
\end{eqnarray*} 
Using the Leibniz rule, and the requirement that $\nabla$ commutes
with tensor contraction, we can extend the action of dynamical covariant derivative to arbitrary tensor fields and (vector valued) forms on $T_{0}M$. The dynamical covariant derivative $\nabla$ preserves the horizontal and vertical distributions $HT_0M$ and $VT_0M$.

For the components of a $(0,2)$-type tensor field $g_{ij}$ and a $1$-form $\omega_i$ on $T_{0}M$, their dynamical covariant derivatives are given by: 
\begin{eqnarray*}
\nabla g_{ij}=S(g_{ij}) - N_i^k g_{kj} - N_j^k g_{ik}, \quad \nabla \omega_i = S(\omega_i) - N_i^k\omega_k.
\end{eqnarray*}

For an arbitrary spray $S$ and an arbitrary smooth function $L$
on $T_0M$, the following $1$-form is a semi-basic $1$-form:
\begin{eqnarray}
\delta_{S}L & =\left\{ S\left(\dfrac{\partial L}{\partial y^{i}}\right)-\dfrac{\partial L}{\partial x^{i}}\right\} dx^{i} & =\left\{ \nabla\left(\frac{\partial L}{\partial y^{i}}\right)-\frac{\delta L}{\delta x^{i}}\right\} dx^{i}.\label{elform}
\end{eqnarray}
It is called the \emph{Euler-Lagrange} $1$-form, or \emph{the Lagrange differential}.

The regularity condition of a Finsler metric
implies that there exists a unique spray $S$ on $T_{0}M$ satisfying
the Euler Lagrange equations \eqref{elf2}, which can be written also using the Euler-Lagrange $1$-form as follows:
\begin{eqnarray*}
\delta_{S}F^{2}=0.
\end{eqnarray*}
This spray is called the \emph{geodesic spray} of the Finsler metric, and
its local coefficients are given by
\begin{eqnarray*}
G^{i}=\frac{1}{4}g^{ik}\left(\frac{\partial^{2}F^{2}}{\partial y^{k}\partial x^{h}}y^{h}-\frac{\partial F^{2}}{\partial x^{k}}\right).
\end{eqnarray*}

Two sprays $S$ and $\widetilde{S}$ are \emph{projectively
related} if their geodesics coincide up to an orientation preserving
reparametrization. A spray $S$ is projectively related to a Finsler structure $\widetilde{F}$ if it is projectively related to its geodesic spray.

Two sprays $S$ and $\widetilde{S}$ are  projectively related if and only if there exists a smooth function
$P$ on $T_{0}M$, positively homogeneous of order $1$, such that
$\widetilde{S}=S-2P\mathcal{C}$. $P$ is called the deformation (projective) factor. 

A Finsler metric $\widetilde{F}$ is projectively related to a spray
$S$ if and only if it satisfies the \emph{Hamel equation}: 
\begin{eqnarray}
\delta_{S}\widetilde{F}=0.\label{eq:Hamel}
\end{eqnarray}
In this case, the geodesic spray $\widetilde{S}$ of $\widetilde{F}$ and
$S$ are related by the projective factor $2P\widetilde{F}=S(\widetilde{F})$. 

Gyroscopic sprays are given by equation \eqref{isomega} that is equivalent to $\delta_S\widetilde{F}=i_S\omega$. This equation generalizes the Hamel equation \eqref{eq:Hamel}. 

\section{Proof of the main theorems.}

In Finsler geometry, it is important to decide whether or not, the integral curves of a system \eqref{sode} are related to the geodesics of a Finsler metric $F$, eventually up to an orientation preserving reparameterisation, and/or under the action of some gyroscopic force. These situations correspond to the following metrizability problems, when a given spray $S$ is:
\begin{itemize}
\item[FM)] the geodesic spray of some Finsler metric;
\item[PM)] projectively related to some Finsler metric;
\item[GM)] the gyroscopic spray of some Finsler metric.
\end{itemize} 
There are various characterisations for these metrisability problems, the corresponding integrability conditions being known as the Helmholtz conditions, \cite{BC15, BD09, KS85, SV02}. The main results of this work, Theorems \ref{thm:fm}, \ref{thm:pm} and \ref{thm:gyroscopic} provide characterisations of the three metrizability problems using the dynamical covariant derivative of the metric tensor \eqref{gij} and the angular tensor \eqref{angular}.

\begin{lem} \label{lemma} Consider $f\in\mathcal{C}^{\infty}(T_{0}M)$
a smooth function and $\sigma=\sigma_{i}(x,y)dx^{i}$ a semi-basic
$1$-form, both homogeneous of order $k\in\mathbb{Z}^{*}$. Then, $f$ satisfies the Euler-Lagrange-type equation $\delta_{S}f=\sigma$ if an only if it satisfies the following two equations:
\begin{eqnarray}
\begin{cases}
& \nabla\left(\dfrac{\partial^{2}f}{\partial y^{i}\partial y^{j}}\right) =\dfrac{1}{2}\left(\dfrac{\partial\sigma_{i}}{\partial y^{j}}+\dfrac{\partial\sigma_{j}}{\partial y^{i}}\right),\vspace{2mm} \\
& \dfrac{\delta}{\delta x^{i}}\left(\dfrac{\partial f}{\partial y^{j}}\right)-\dfrac{\delta}{\delta x^{j}}\left(\dfrac{\partial f}{\partial y^{i}}\right)  =\dfrac{1}{2}\left(\dfrac{\partial\sigma_{i}}{\partial y^{j}}-\dfrac{\partial\sigma_{j}}{\partial y^{i}}\right).
\end{cases}\label{dsfs}
\end{eqnarray}
\end{lem}
\begin{proof} According to Euler's theorem, a $k$-homogeneous function, $k\in\mathbb{Z}^{*}$, is well determined by its fibre derivatives: $kf=y^i{\partial f}/{\partial y^i}$. For the Euler-Lagrange semi-basic $1$-form $\delta_Sf$, we denote its components $\omega_i= S\left({\partial f}/{\partial y^{i}}\right)-{\partial f}/{\partial x^{i}}$. Therefore, we have the Euler-Lagrange-type equations $\delta_{S}f=\sigma$ if and only if their $k$-homogeneous components are equal, which is equivalent to: ${\partial\omega_{i}}/{\partial y^{j}}={\partial\sigma_{i}}/{\partial y^{j}}$, $\forall i,j\in\overline{1,n}$. Using the commutation rules for ${\partial}/{\partial y^j}$ and $S$ we have:
\begin{align*}
& \frac{\partial\omega_{i}}{\partial y^{j}} = S\left(\dfrac{\partial^{2}f}{\partial y^{j}\partial y^{i}}\right) + \dfrac{\partial^2 f}{\partial x^j\partial y^i}- 2N^k_j \frac{\partial^2 f}{\partial y^k \partial y^i} - \dfrac{\partial^2 f}{\partial y^j \partial x^i} \\
 & =\nabla\left(\frac{\partial^{2}f}{\partial y^{j}\partial y^{i}}\right)+\frac{\delta}{\delta x^{j}}\left(\frac{\partial f}{\partial y^{i}}\right)-\frac{\delta}{\delta x^{i}}\left(\frac{\partial f}{\partial y^{j}}\right).
\end{align*}
Hence ${\partial\omega_{i}}/{\partial y^{j}}={\partial\sigma_{i}}/{\partial y^{j}}$ if and only if the symmetric part and the skew symmetric part
of the right and left sides of these partial derivatives coincide. This is equivalent with the two conditions \eqref{dsfs}. 
\end{proof}

\subsection*{Proof of Theorem \ref{thm:fm} }

A given spray $S$ is the geodesic spray of some Finsler structure
$F$ if it satisfies the Euler-Lagrange equation $\delta_{S}F^{2}=0$. We use Lemma \eqref{lemma} for $f=F^2$ and $\sigma=0$, both homogeneous of order $2$. Then, the first equation \eqref{dsfs} reads
$\nabla g_{ij} = 0$. 
Therefore, the condition \eqref{ngij}, $\nabla g_{ij}=0$, is a necessary condition for a spray to be Finsler metrizable. We will prove now that this condition is also sufficient.

We assume that the condition \eqref{ngij} is true. Since the dynamical covariant derivative $\nabla$ commutes with the contractions and satisfies the Leibniz rule we have the following implications:
\begin{eqnarray*}
\nabla g_{ij}=0\Rightarrow \nabla(g_{ij}y^j)=0 \Rightarrow \nabla\left(\dfrac{\partial F^2}{\partial y^i}\right)=0 \Rightarrow \nabla\left(\dfrac{\partial F^2}{\partial y^i}y^i\right)=0 \Rightarrow \nabla\left(F^2\right)=0 \Rightarrow S\left( F^2 \right)=0.
\end{eqnarray*}
If we differentiate with respect to $y^{j}$ the last relation and use the commutation rule for $S$ and $\partial/\partial y^j$ we obtain:
\begin{eqnarray*}
S\left(\dfrac{\partial F^2}{\partial y^j}\right) + \dfrac{\partial F^2}{\partial x^j} - 2N^k_j \dfrac{\partial F^2}{\partial y^k}=0 \Rightarrow \nabla\left(\dfrac{\partial F^2}{\partial y^j}\right) + \dfrac{\delta F^2}{\delta x^j}=0 \Rightarrow \dfrac{\delta F^2}{\delta x^j}=0. 
\end{eqnarray*}
Using these computations and the second expression for the Euler-Lagrange $1$-form \eqref{elform}, we have
\begin{eqnarray*}
\delta_SF^2=\left\{\nabla\left(\dfrac{\partial F^2}{\partial y^i}\right) - \dfrac{\delta F^2}{\delta x^i}\right\}dx^i=0,
\end{eqnarray*}
and hence $S$ is the geodesic spray of the Finsler metric $F$.

\subsection*{Proof of Theorem \ref{thm:pm}}

For the first part of the Theorem \ref{thm:pm}, consider $S$ a spray and a Finsler metric $\widetilde{F}$. The spray $S$ is projectively related to $\widetilde{F}$ if and only if there exists a $1$-homogeneous function $P\in C^{\infty}(T_0M)$ such that $\delta_S\widetilde{F}^2=2Pd_J\widetilde{F}^2$, \cite[Theorem 3.2]{BC20}. 

We assume that the spray $S$ is projectively related to the Finsler metric $\widetilde{F}$ and we prove that the dynamical covariant derivative of the metric tensor $\widetilde{g}$ satisfies the Levi-Civita equations \eqref{npgij}. We use again Lemma \ref{lemma} for $f=F^2$ and $\sigma=2Pd_JF^2$, both homogeneous of order $2$.

First, we compute the right hand side of the first equation \eqref{dsfs} for $\sigma_i=2P{\partial \widetilde{F}^2}/{\partial y^i}$:
\begin{eqnarray*}
\dfrac{\partial \sigma_i}{\partial y^j}=2\dfrac{\partial \widetilde{F}^2}{\partial y^i}\dfrac{\partial P}{\partial y^j} + 4P \widetilde{g}_{ij}.
\end{eqnarray*}
It follows from the first formula \eqref{dsfs} that we have:
\begin{eqnarray*}
\nabla\widetilde{g}_{ij}=2P\widetilde{g}_{ij}+\dfrac{\partial P}{\partial y^{i}}\widetilde{g}_{kj}y^{k}+\dfrac{\partial P}{\partial y^{j}}\widetilde{g}_{ik}y^{k}
\end{eqnarray*}
and hence the metric tensor $\widetilde{g}_{ij}$ satisfy the Levi-Civita equations \eqref{npgij}.

For the converse implication, we consider a spray $S$ and a Finsler structure $\widetilde{F}$, whose metric tensor $\widetilde{g}_{ij}$ satisfies the Levi-Civita equations \eqref{npgij}. If we contract both sides of this equation by $y^j$ and use the homogeneity of the involved quantities, we obtain:
\begin{eqnarray}
\nabla\left(\dfrac{\partial \widetilde{F}^2}{\partial y^i}\right) = 3P \dfrac{\partial \widetilde{F}^2}{\partial y^i} + 2 \widetilde{F}^2 \dfrac{\partial P}{\partial y^i}. \label{npfy}
\end{eqnarray}
If we contract again both sides of this equation by $y^i$, we obtain 
$\nabla (\widetilde{F}^2)=4P\widetilde{F}^2$ and hence $S(\widetilde{F}^2)=4P\widetilde{F}^2$. If we take the derivatives of both sides of the last equation with respect to $y^i$ and use the commutation rule for $S$ and ${\partial}/{\partial y^i}$ we obtain 
\begin{eqnarray*}
S\left(\dfrac{\partial \widetilde{F}^2}{\partial y^i}\right) + \dfrac{\partial \widetilde{F}^2}{\partial x^i} - 2N^k_i \dfrac{\partial \widetilde{F}^2}{\partial y^k} = 4P \dfrac{\partial \widetilde{F}^2}{\partial y^i} + 4 \widetilde{F}^2 \dfrac{\partial P}{\partial y^i},
\end{eqnarray*}    
which can be written as follows
\begin{eqnarray*}
\nabla\left(\dfrac{\partial \widetilde{F}^2}{\partial y^i}\right) + \dfrac{\delta \widetilde{F}^2}{\delta x^i}  = 4P \dfrac{\partial \widetilde{F}^2}{\partial y^i} + 4 \widetilde{F}^2 \dfrac{\partial P}{\partial y^i}.
\end{eqnarray*} 
From this formula, if we use \eqref{npfy}, we obtain
\begin{eqnarray}
\dfrac{\delta \widetilde{F}^2}{\delta x^i}  = P \dfrac{\partial \widetilde{F}^2}{\partial y^i} +  \widetilde{F}^2 \dfrac{\partial P}{\partial y^i}. \label{dpfy}
\end{eqnarray}
If we subtract the two formulae \eqref{npfy} and \eqref{dpfy}, side by side, we obtain 
\begin{eqnarray*}
\nabla\left(\dfrac{\partial \widetilde{F}^2}{\partial y^i}\right) - \dfrac{\delta \widetilde{F}^2}{\delta x^i}  = 2P \dfrac{\partial \widetilde{F}^2}{\partial y^i}, 
\end{eqnarray*}  
which means $\delta_S\widetilde{F}^2=2Pd_J\widetilde{F}^2$ and hence, according to \cite[Theorem 3.2]{BC20}, the spray $S$ is projectively related to the Finsler metric $\widetilde{F}$.

The second part of Theorem \ref{thm:pm} is a direct consequence of Lemma \ref{lemma}, for $f=F$ and $\sigma=0$.

A spray $S$ is projectively related to a Finsler structure $\widetilde{F}$ if and only if $\delta_{S}\tilde{F}=0$. Then, according to the first formula \eqref{dsfs}, we obtain that 
\begin{eqnarray*}
\nabla\left(\frac{\partial^{2}\widetilde{F}}{\partial y^{i}\partial y^{j}}\right) =0.
\end{eqnarray*}
Hence, the geodesic invariance $\nabla\left(\widetilde{h}_{ij}/\widetilde{F}\right)=0$ is
a necessary condition for $S$ to be projectively related to $\widetilde{F}$, where $\widetilde{h}_{ij}$ 
is the angular metric \eqref{angular} of $\widetilde{F}$.

\subsection*{Proof of Theorem \ref{thm:gyroscopic}}

Theorem \ref{thm:gyroscopic} states that the most general class of sprays that satisfy the geodesic invariance condition \eqref{nhij2} is the class of gyroscopic sprays.

Consider $S$ a spray, a Finsler metric $\widetilde{F}$ and a basic $2$-form $\omega\in \Lambda^2(M)$. According to Lemma \ref{lemma} for $f=F$ and $\sigma=i_S\omega$, homogeneous of order $1$, we have that the spray $S$ is gyroscopic if and only if the following two equations are satisfied:
\begin{eqnarray}
\begin{cases}
\nabla\left(\dfrac{\partial^{2}\widetilde{F}}{\partial y^{i}\partial y^{j}}\right) & =0,\\
\dfrac{\delta}{\delta x^{i}}\left(\dfrac{\partial \widetilde{F}}{\partial y^{j}}\right)-\dfrac{\delta}{\delta x^{j}}\left(\dfrac{\partial\widetilde{F}}{\partial y^{i}}\right) & =\omega_{ij}(x).
\end{cases} \label{cgyro}
\end{eqnarray}
Therefore, the first equation \eqref{cgyro} is a necessary condition for a spray $S$ to be gyroscopic.

We will prove now that this condition is sufficient as well. For this, we will show that the first condition \eqref{cgyro} implies the second condition \eqref{cgyro}. We start with the first condition \eqref{cgyro}, which can be written as follows:
\begin{eqnarray*}
 S\left(\frac{\partial^{2}\widetilde{F}}{\partial y^{i}\partial y^{j}}\right)-N_{i}^{l}\frac{\partial^{2}\widetilde{F}}{\partial y^{l}\partial y^{j}}-N_{j}^{l}\frac{\partial^{2}\widetilde{F}}{\partial y^{i}\partial y^{l}}=0.
\end{eqnarray*}
We differentiate with respect to $y^{k}$ and use again the commutation rule for $S$ and ${\partial}/{\partial y^k}$:
\begin{eqnarray*}
& & S\left(\dfrac{\partial^{3}\widetilde{F}}{\partial y^{i}\partial y^{j}\partial y^k}\right) + \dfrac{\partial}{\partial x^k} \left(\dfrac{\partial^2 \widetilde{F}}{\partial y^i \partial y^j} \right) - 2N_{k}^{l}\dfrac{\partial^3 \widetilde{F}}{\partial y^i \partial y^j \partial y^l}  \\ & - & \dfrac{\partial^2 G^l}{\partial y^i \partial y^k} \dfrac{\partial^2 \widetilde{F}}{\partial y^j \partial y^l} - N_{i}^{l}\dfrac{\partial^3 \widetilde{F}}{\partial y^j \partial y^k \partial y^l} - \dfrac{\partial^2 G^l}{\partial y^j \partial y^k} \dfrac{\partial^2 \widetilde{F}}{\partial y^i \partial y^l} - N_{j}^{l}\dfrac{\partial^3 \widetilde{F}}{\partial y^i \partial y^k \partial y^l} = 0. 
\end{eqnarray*}
The left hand side of the above equation is a $(0,3)$-type tensor. Of this tensor, we will consider only its skew-symmetric part with respect to $j$ and $k$, and hence we obtain the following equation:
\begin{eqnarray*}
& \dfrac{\partial}{\partial y^i}\left(\dfrac{\partial}{\partial x^k}\left(\dfrac{\partial \widetilde{F}}{\partial y^j}\right) - N^l_k \dfrac{\partial^2 \widetilde{F}}{\partial y^j\partial y^l} - \dfrac{\partial}{\partial x^j}\left(\dfrac{\partial \widetilde{F}}{\partial y^k}\right) + N^l_j \dfrac{\partial^2 \widetilde{F}}{\partial y^k\partial y^l} \right)= \\
& \dfrac{\partial}{\partial y^i}\left(\dfrac{\delta}{\delta x^k}\left(\dfrac{\partial \widetilde{F}}{\partial y^j}\right)  - \dfrac{\delta}{\delta x^j}\left(\dfrac{\partial \widetilde{F}}{\partial y^k}\right) \right)=\dfrac{\partial \omega_{kj}}{\partial y^i} = 0. 
\end{eqnarray*}
It follows that 
\begin{eqnarray*}
\omega_{kj}=\dfrac{\delta}{\delta x^k}\left(\dfrac{\partial \widetilde{F}}{\partial y^j}\right)  - \dfrac{\delta}{\delta x^j}\left(\dfrac{\partial \widetilde{F}}{\partial y^k}\right)
\end{eqnarray*}
are the components of a basic $2$-form, $\omega=\omega_{kj}(x)dx^k\wedge dx^j$. Therefore, the second equation  \eqref{cgyro} is satisfied. According to Lemma \ref{lemma}, the two equations \eqref{cgyro} are equivalent to the fact that the spray $S$ satisfies the gyroscopic equation $\delta_S\widetilde{F}= i_S\omega$.

\subsection*{Acknowledgements} 
This manuscript has no associated data.


\begin{thebibliography}{10}
\bibitem{Anastasiei03} {Anastasiei, Mihai}: \emph{Metrizable linear
connections in vector bundles}, Publ. Math. Debrecen, 62 (2003), no.
3--4, 277--287.

\bibitem{Berwald41} {Berwald, Ludwig}: \emph{On Finsler and Cartan
geometries. III. Two-dimensional Finsler spaces with rectilinear extremals},
Ann. of Math., \textbf{42} (1) (1941), 84--112.

\bibitem{Berwald36} {Berwald, Ludwig}: \emph{On the projective geometry
of paths}, Ann. of Math. (2) 37 (1936), no. 4, 879--898.

\bibitem{BDE09} {Bryant, Robert; Dunajski, Maciej; Eastwood, Michael}:
\emph{Metrisability of two-dimensional projective structures}, J.
Differential Geom. 83 (2009), no. 3, 465--499.

\bibitem{Bucataru22} {Bucataru, Ioan,}: \emph{Invariant volume forms
and first integrals for geodesically equivalent Finsler metrics},
Proc. Amer. Math. Soc. 150 (2022), no. 10, 4475--4486.

\bibitem{BC20} {Bucataru, Ioan;  Cre\c{t}u, Georgeta}: \emph{A Characterisation for Finsler Metrics of Constant Curvature and a Finslerian Version of Beltrami Theorem},
The Journal of Geometric Analysis, \textbf{30}(2020), no. 1, 617--631.


\bibitem{BC15} {Bucataru, Ioan; Constantinescu, Oana}: \emph{Generalized
Helmholtz conditions for non-conservative Lagrangian systems}, Math.
Phys. Anal. Geom. 18 (2015), no. 1, Art. 25, 24 pp.

\bibitem{BCD11} {Bucataru, Ioan; Constantinescu, Oana; Dahl, Matias
F.}: \emph{A geometric setting for systems of ordinary differential
equations}, Int. J. Geom. Methods Mod. Phys. 8 (2011), no. 6, 1291--1327.

\bibitem{BD09} {Bucataru, Ioan; Dahl, Matias F.}: \emph{Semi-basic 1-forms
and Helmholtz conditions for the inverse problem of the calculus of
variations}, J. Geom. Mech., \textbf{1}(2) (2009), 159--180.

\bibitem{Grifone72} {Grifone, Joseph}: \emph{Structure presque-tangente
et connexions I}, Ann. Inst. Fourier, \textbf{22} (1972), 287--334.

\bibitem{GM00} {Grifone, Joseph; Muzsnay, Zoltan}: \emph{Variational Principles for Second Order Differential Equations. Application of the Spencer
Theory to Characterize Variational Sprays}, World Scientific, 2000.

\bibitem{Kowalski88} {Kowalski, Oldrich}: \emph{Metrizability of
affine connections on analytic manifolds}, Note Mat. 8 (1988), no.
1, 1--11.

\bibitem{KS85} {Krupka, Demeter; Sattarov, Abdurasoul Ezbekhovich}:
\emph{The inverse problem of the calculus of variations for Finsler
structures}, Math. Slovaca 35 (1985), no. 3, 217--222.

\bibitem{MVV08} {Mikes, Josef; Kiosak, Volodymyr; Vanzurova, Alena}:
\emph{Geodesic mappings of manifolds with affine connection}, Palacky
University Olomouc, Olomouc, 2008. 222 pp.

\bibitem{Schmidt73} {Schmidt, B. G.}: \emph{Conditions on a connection
to be a metric connection}, Comm. Math. Phys. 29 (1973), 55--59.

\bibitem{Shen01} {Shen, Zhongmin}: \emph{Differential geometry of spray
and Finsler spaces}, Springer, 2001.

\bibitem{SLK14} {Szilasi, J\'{o}zsef; Lovas,  Rezs\"{o}; Kert\'esz, D\'{a}vid}: \emph{Connections,
sprays and Finsler structures}, World Scientific, 2014.

\bibitem{SV02} {Szilasi, J\'{o}zsef; Vattamany, Szabolcks}: \emph{On the Finsler-metrizabilities of spray manifolds}, Period. Math. Hungar. 44 (2002), no. 1, 81--100

\bibitem{TK12} {Tanaka, Erico; Krupka, Demeter}: \emph{On metrizability
of invariant affine connections}, Int. J. Geom. Methods Mod. Phys.
9 (2012), no. 1, 1250014, 15 pp.\end{thebibliography}
\end{document}